\documentclass[12pt]{article}
\usepackage[a4paper,margin=1in]{geometry}

\usepackage{tikz-cd,amsmath,amsthm,amssymb,mathtools,stmaryrd}

\newcommand{\lowershriek}{_!}
\newcommand{\upperstar}{^{\raisebox{-0.25ex}[0ex][0ex]{\(\ast\)}}}

\def\overarrow#1{{\vec{#1}}}
\def\nondeg{\overarrow}

\newcommand{\Q}{\mathbb{Q}}

\newcommand{\Phieven}{\Phi_{\text{\rm even}}}
\newcommand{\Phiodd}{\Phi_{\text{\rm odd}}}
\renewcommand{\epsilon}{\varepsilon}

\DeclareRobustCommand\shortsetminus{%
  \mathchoice%
    {\kern0pt\raise0.25ex\hbox{$\displaystyle \mathbb r$}\kern0.8pt}
    {\kern0pt\raise0.28ex\hbox{$\textstyle \mathbb r$}\kern0.8pt}
    {\kern0pt\raise0.20ex\hbox{$\scriptstyle \mathbb r$}\kern0.4pt}
    {\kern0pt\raise0.17ex\hbox{$\scriptscriptstyle \mathbb r$}\kern0.2pt}
}%

\newcommand{\inertto}{\rightarrowtail}
\newcommand{\into}{\hookrightarrow}
\newcommand{\actto}{\rightarrow\Mapsfromchar}

\tikzset{
  act /.tip = >|
}

\tikzset{
    onto/.style={/tikz/commutative diagrams/twoheadrightarrow}
}

\tikzset{
    into/.style={/tikz/commutative diagrams/hookrightarrow}
}

\tikzset{
    inertto/.style={/tikz/commutative diagrams/rightarrowtail}
}

\tikzset{
    rot/.style={shift={(-4.5pt,0pt)}, rotate=-45}
}

\newcommand{\drpullback}
  {\arrow[phantom]{dr}[very near start,description]{\lrcorner}}
\newcommand{\dlpullback}
  {\arrow[phantom]{dl}[very near start,description]{\llcorner}}

\usepackage{eucal}

\DeclareMathAlphabet{\mathbbe}{U}{bbold}{m}{n}
\newcommand{\simplexcategory}{\mathbbe{\Delta}}

\DeclareSymbolFont{bbold}{U}{bbold}{m}{n}
\DeclareSymbolFontAlphabet{\mathbbold}{bbold}

\newtheorem{lemma}{Lemma}[subsection]
\newtheorem{prop}[lemma]{Proposition}
\newtheorem{cor}[lemma]{Corollary}
\newtheorem{theorem}[lemma]{Theorem}
\theoremstyle{definition}
\newtheorem{remark}[lemma]{Remark}
\newtheorem{example}[lemma]{Example}

\newtheorem{taller}[lemma]{$\!\!$}

\newenvironment{blanko}[1]%
{\begin{taller}{\normalfont\bfseries  #1}\normalfont}%
{\end{taller}}

\makeatletter
\renewcommand{\tableofcontents}{%
   \begin{center}
\begin{minipage}{100mm}
    
   \begin{center}
     
       \bf{\contentsname}
   \end{center}
	
       \vspace{-18pt}
	
   \footnotesize
   \begin{center}
\@starttoc{toc}
   \end{center}	
\end{minipage}
	\end{center}
	\addvspace{2em \@plus\p@}
}

\makeatother

\newcommand{\spaces}{\mathcal{S}}

\newcommand{\id}{\operatorname{id}}

\newcommand{\name}[1]{\ulcorner #1\urcorner}
\newcommand{\coname}[1]{\raisebox{-1.5pt}{$\llcorner$}\kern0.8pt#1\kern0.8pt\raisebox{-1.5pt}{$\lrcorner$}}

\newcommand{\op}{^{\text{{\rm{op}}}}}

\newcommand{\N}{\mathbb{N}}

\newcommand{\Decbot}[1]{\operatorname{Dec}_\bot{}\kern-2pt{#1}}
\newcommand{\Dectop}[1]{\operatorname{Dec}_\top{}\kern-2pt{#1}}

\providecommand{\kat}[1]{\text{\textbf{\textsl{#1}}}}

\usepackage[hang,flushmargin]{footmisc}

\newcommand\blfootnote[1]{%
  \begingroup
  \renewcommand\thefootnote{}\footnote{#1}%
  \addtocounter{footnote}{-1}%
  \endgroup
}

\usepackage[pdfencoding=auto,raiselinks=false,colorlinks=true,allcolors=.,bookmarks=false]{hyperref}

\begin{document}
\title{Convex decomposition spaces and \\[4pt]
Crapo complementation formula}

\author{Imma G\'alvez-Carrillo, Joachim Kock, and Andrew Tonks}

\date{}

\maketitle

\begin{abstract}
  We establish a Crapo complementation formula for the M\"obius function
  $\mu^X$ in a general decomposition space $X$ in terms of a convex
  subspace $K$ and its complement: $\mu^X \simeq \mu^{X\shortsetminus K} +
  \mu^X*\zeta^K*\mu^X$. We work at the objective level, meaning that the
  formula is an explicit homotopy equivalence of $\infty$-groupoids. Almost
  all arguments are formulated in terms of (homotopy) pullbacks. Under
  suitable finiteness conditions on $X$, one can take homotopy cardinality
  to obtain a formula in the incidence algebra at the level of
  $\Q$-algebras. When $X$ is the nerve of a locally finite poset, this recovers the Bj\"orner--Walker
  formula, which in turn specialises to the original Crapo complementation
  formula when the poset is a finite lattice.
  A substantial part of the work is to introduce and develop the notion of
  convexity for decomposition spaces, which in turn requires some general
  preparation in decomposition-space theory, notably some results on
  reduced covers and ikeo and semi-ikeo maps. These results may be of wider
  interest. Once this is set up, the objective proof of the Crapo formula is quite
  similar to that of Bj\"orner--Walker.
\end{abstract}

\blfootnote{
		\raisebox{-0.1pt}{\includegraphics[height=1.65ex]{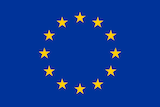}}
		This work has received funding from the European Union's Horizon 2020 research
        and innovation programme under Marie Sk\l odowska-Curie grant agreement 
        No.~101028099 and a Spanish university requalification and mobility 
		grant (UP2021-034, UNI/551/2021) with NextGenerationEU funds. More
		indirectly it was supported by research grants PID2019-103849GB-I00, 
		PID2020-116481GB-I00,
        and
        PID2020-117971GB-C22
        (AEI/FEDER, UE) of Spain, grants
		2021-SGR-0603 and 2021-SGR-1015 of Catalonia, and was also supported through 
		the Severo Ochoa and Mar\'ia de Maeztu Program for Centers and Units 
		of Excellence in R\&D grant number CEX2020-001084-M.}

\small

\vspace*{-14pt}

\tableofcontents

\vspace*{-14pt}

\section{Introduction}

The theory of incidence algebras and M\"obius inversion for locally finite 
posets was developed by 
Rota~\cite{Rota:Moebius} (see also Joni--Rota~\cite{JoniRotaMR544721}).
Leroux~\cite{Leroux:1976}, \cite{Content-Lemay-Leroux} showed how the theory can be generalised from 
locally finite posets to certain locally finite categories called {\em 
M\"obius categories}.
However, beyond the basic constructions, the theory did not develop much
for some decades.
An important development, independent of Leroux theory, was the
simplicial viewpoint taken by D\"ur~\cite{Dur:1986}. Next, an important
step was the objective viewpoint of Lawvere and Menni~\cite{LawvereMenniMR2720184},
upgrading algebraic identities to bijections of sets and equivalences of 
groupoids.

The present authors \cite{Galvez-Kock-Tonks:1512.07573},
\cite{Galvez-Kock-Tonks:1512.07577} (see also
\cite{Galvez-Kock-Tonks:1612.09225}) introduced the notion of decomposition
space (the same thing as the $2$-Segal spaces of Dyckerhoff and
Kapranov~\cite{Dyckerhoff-Kapranov:1212.3563}) as a general framework for
incidence algebras and M\"obius inversion. Following the direction set out
by Leroux~\cite{Leroux:1976}, \cite{Content-Lemay-Leroux}, the theory is
categorical, and in fact $\infty$-categorical. A benefit of this 
homotopical viewpoint is that
symmetries are built in, which is useful even in classical combinatorial
situations that do not have any $\infty$-category appearance. Following the
direction of D\"ur~\cite{Dur:1986} the theory is simplicial and covers a
class of simplicial $\infty$-groupoids which are not Segal spaces. This allows
many combinatorial co-, bi- and Hopf
algebras to be realised as
incidence coalgebras of decomposition spaces while they are not incidence 
coalgebras of posets or categories. Finally following
the direction set out by Lawvere and Menni~\cite{LawvereMenniMR2720184}, the theory
is objective (with the link to ordinary algebra over $\Q$ given by
homotopy cardinality). In particular, remarkably many arguments
can be formulated in terms of (homotopy) pullbacks.
One benefit of the objective approach is that
many formulae can be established without imposing finiteness conditions:
they are still valid homotopy equivalences of $\infty$-groupoids. The
finiteness conditions are required only to be able to take cardinality.

With the new toolbox at hand it is now an overall programme to upgrade the 
classical theory from posets to decomposition spaces, and investigate new 
applications.
Beyond the general theory, an important extension of Rota's original 
contribution was the formula of Carlier~\cite{Carlier:1801.07504} for 
the relationship between the 
M\"obius function of two decomposition spaces related by an 
$\infty$-adjunction, which generalises Rota's formula for a Galois 
correspondence of posets.

In the present paper we give a generalisation to the decomposition-space
setting of another classical 
formula, namely Crapo's complementation formula, originally formulated
in the setting of lattices~\cite{Crapo} but generalised to arbitrary posets
by Bj\"orner and Walker~\cite{Bjoerner-Walker}.

To do so, we first have to develop some general theory on convex 
subspaces of a decomposition space, and some general results about ikeo 
and semi-ikeo maps.

Functoriality is an important aspect of the objective approach to incidence
algebras. {\em Culf} maps between decomposition spaces induce algebra
homomorphisms contravariantly on incidence algebras, whereas {\em ikeo} maps
induce algebra homomorphisms covariantly on incidence algebras. Culf maps
have been exploited a lot already both in the original series of papers
\cite{Galvez-Kock-Tonks:1512.07573,Galvez-Kock-Tonks:1512.07577,Galvez-Kock-Tonks:1512.07580}
and in later works (see notably \cite{Hackney-Kock:2210.11191}). Ikeo maps
have not yet received the same attention, and our first task is to develop
some basic theory about them needed for the Crapo formula. While the culf
condition interacts very nicely with the original characterisation of decomposition
spaces in terms of active-inert pullbacks, the ikeo condition interacts
better with an alternative characterisation of decomposition spaces in
terms of pullbacks with inert covers (to be made precise below), so we
take the opportunity to develop that viewpoint 
(cf.~Theorem~\ref{thm:decomp}).

A subtle issue 
is the preservation of units for the convolution product in the 
incidence algebras. While for the decomposition-space axioms unitality has 
turned out to be 
automatic~\cite{Feller-Garner-Kock-Proulx-Weber:1905.09580} (the incidence
algebra of a simplicial set is automatically unital if just it is 
associative), and while the contravariant functoriality in simplicial  maps
preserves units automatically if it preserves the convolution product, the 
same is not true for the covariant functoriality: there are simplicial maps
that are not quite ikeo, which preserve the convolution product without 
preserving the unit. Reluctantly we call them {\em semi-ikeo}.
We show that full inclusions are such maps. We show that
if a simplicial space is semi-ikeo over a decomposition space then it is
itself a decomposition space (Lemma~\ref{semiikeo / decomp is decomp}).

A full inclusion of simplicial spaces is called {\em convex} when it is
furthermore culf. A convex subspace of a decomposition space is thus again
a decomposition space, and its complement is a decomposition space too 
(although of course not generally convex).

With these preparations we are ready to state and prove the Crapo
complementation formula for decomposition spaces: for an arbitrary
decomposition space $X$ and a convex subspace $K$, we have the following
formula (Theorem~\ref{crapo}) relating the M\"obius function of $X$ with
that of $K$ and its complement:
$$
\mu^X = \mu^{X\shortsetminus K} + \mu^X *\zeta^K *\mu^X  .
$$

The statement here involves formal differences, since each M\"obius 
function is an alternating sum, but after moving all negative terms to the 
other side of the equation, the formula is established as an explicit 
homotopy equivalence of $\infty$-groupoids. The formula determines
$\mu^X$ from $\mu^{X\shortsetminus K}$ and $\zeta^K$ by a well-founded 
recursion expressed by the convolution product.

\section{Decomposition spaces}
\label{Sec:decomp}

The main contribution of this section is the characterisation of decomposition spaces
in terms of squares of reduced covers against active injections (Conditions
(3) and (4) in Theorem~\ref{thm:decomp} below). This 
condition plays well together with semi-ikeo maps, as we shall see in  
Section~\ref{sec:ikeo}.

\begin{blanko}{Active and inert maps.}
  The simplex category $\simplexcategory$ (whose objects are the nonempty
  finite ordinals $[n]$ and whose morphisms are the monotone maps)
  has an active-inert factorisation system.
  An arrow in $\simplexcategory$ is \emph{active}, written $a: [m]\actto [n]$, when it preserves
  end-points, $a(0)=0$ and $a(m)=n$; it is \emph{inert}, written $a: [m]\inertto [n]$, if it is distance
  preserving, $a(i+1)=a(i)+1$ for $0\leq i\leq m-1$.
  The active maps are generated by the codegeneracy maps $s^i : [n+1] \actto [n]$ and by the {\em
  inner} coface maps $d^i : [n-1] \actto [n]$, $0 < i < n$, while the inert maps are
  generated by the {\em outer} coface maps $d^\bot := d^0$ and $d^\top:= d^n$.
  Every morphism in $\simplexcategory$ factors uniquely as an active map
  followed by an inert map.  Furthermore, it is a basic fact 
  \cite[Lemma 2.7]{Galvez-Kock-Tonks:1512.07573}
  that active and inert maps in $\simplexcategory$ admit pushouts along each
  other, and the resulting maps are again active and inert.
\end{blanko}

\begin{blanko}{Decomposition spaces \cite{Galvez-Kock-Tonks:1512.07573}.}
  A simplicial space $X: \simplexcategory\op\to\spaces$ is called a {\em
  decomposition space} when it takes active-inert pushouts to pullbacks.
  It has turned out \cite{Feller-Garner-Kock-Proulx-Weber:1905.09580}
  that the degeneracy maps are not required among the 
  active maps to state the condition, so to check the decomposition-space 
  axioms, it is enough to check the following squares for all $0<i<n$:
 
  \[
  \begin{tikzcd}[column sep={4.5em,between origins}, row sep={3.5em,between origins}]
  X_{1+n} \drpullback \ar[d, -act, "d_{1+i}"'] \ar[r, inertto, "d_{\bot}"] & X_n 
  \ar[d, -act, "d_i"]  \\
  X_n \ar[r, inertto, "d_{\bot}"'] & X_{n-1}
  \end{tikzcd}
  \qquad
  \begin{tikzcd}[column sep={4.5em,between origins}, row sep={3.5em,between origins}]
  X_{n+1} \drpullback \ar[d, -act, "d_{i}"'] \ar[r, inertto, "d_{\top}"] & X_n 
  \ar[d, -act, "d_i"]  \\
  X_n \ar[r, inertto, "d_{\top}"'] & X_{n-1}
  \end{tikzcd}
  \]
  As is custom, we use the words (and symbols) `active' and `inert'
  also for their images in $\spaces$ under a functor $X: 
  \simplexcategory\op\to\spaces$.
\end{blanko}

Since the decomposition-space axiom is formulated in terms of
pullbacks --- as are the notions of culf, ikeo, semi-ikeo, fully faithful,
convex, and convolution product featured in this work --- the following
simple lemma becomes an indispensable tool (used a dozen times in this
paper):

\begin{lemma}[Prism Lemma]
In a prism diagram
  \[
  \begin{tikzcd}[cramped]
  \cdot \ar[r] \ar[d] & \cdot \drpullback \ar[r] \ar[d] & \cdot \ar[d]  \\
  \cdot \ar[r] & \cdot \ar[r] & \cdot
  \end{tikzcd}
  \]
  the left-hand square is a pullback if and only if the whole rectangle is
  a pullback.
\end{lemma}

\subsection{Decomposition spaces from the inert viewpoint}

  We work towards an alternative characterisation of decomposition spaces,
  but first we need to set up some terminology.

\bigskip

For each $[k]\in \simplexcategory$ there are $k$ inert maps 
$$
\rho_i : [1] \inertto [k] \qquad \qquad i = 1,\ldots,k  ,
$$
namely picking out the principal edge $(i-1,i)$. For $k=0$ there are zero such
maps.

\begin{blanko}{Special reduced-cover squares.}
  For an active map $\alpha: [k] \actto [n]$, write $[n_i]$ for the 
  ordinal $[\alpha(i)-\alpha(i-1)]$ appearing in the active-inert factorisation of $\alpha
  \circ \rho_i$:
  \[
  \begin{tikzcd}
  {[1]} \ar[r, inertto, "\rho_i"] \ar[d, dotted, -act, "\alpha_i"']
  & {[k]} \ar[d, -act, "\alpha"]
  \\
  {[n_i]}  \ar[r, inertto, dotted, "\gamma^\alpha_i"'] & {[n]}  .
  \end{tikzcd}
  \]
  If $k>0$, the maps $\gamma^\alpha_i$ together constitute a {\em cover} of $[n]$,
  meaning that they are jointly surjective. A cover is called {\em reduced}
  if no edges are hit twice (for $(\gamma^\alpha_i)$ this is clear) and
  if there are no copies of $[0]$ involved (which
  is the case when $\alpha$ is injective). The notions of cover and
  reduced cover in the inert part of $\simplexcategory$ were first studied 
  by Berger~\cite{Berger:Adv2002}, including the important characterisation of 
  categories: a simplicial set is a category if and only if it is a sheaf 
  for this notion of cover.
  
  The maps $\alpha_i : [1] \actto [n_i]$
  together constitute the unique join decomposition of $\alpha$ into active 
  maps with domain $[1]$: we have
  $$
  \alpha = \alpha_1 \vee \cdots \vee \alpha_k .
  $$
  
  The $k$-tuple of maps $\gamma^\alpha_i$ (and the $k$-tuple of squares)
  thus define for any simplicial space $X$ a diagram
  \begin{equation*}\label{alpha-pbk}\tag{SRCS}
  \begin{tikzcd}[column sep={10em,between origins}]
  X_1 \times \cdots \times X_1 & \ar[l, 
  "{(\rho_1,\ldots,\rho_k)\upperstar}"']   X_k   \\
  X_{n_1} \times \cdots \times X_{n_k} 
  \ar[u, -act, "(\alpha_1\times\cdots\times\alpha_k)\upperstar"] & X_n .
  \ar[u, -act, "{\alpha\upperstar} = (\alpha_1\vee\cdots\vee\alpha_k)\upperstar"'] 
  \ar[l, "{(\gamma^\alpha_1,\ldots,\gamma^\alpha_k)\upperstar}"]
  \end{tikzcd}
  \end{equation*}
  We refer to these squares as {\em special reduced-cover squares}. 
  Note that the vertical maps are active, or products of active 
  maps, while the components of the horizontal maps are inert.
  Here and in the text below we use notation
  such as $(\alpha_1\times\cdots\times\alpha_k)\upperstar$
    and $(\rho_1,\ldots,\rho_k)\upperstar$
    for
    $\alpha_1\upperstar\times\cdots\times\alpha_k\upperstar$
      and $(\rho_1\upperstar,\ldots,\rho_k\upperstar)$ respectively.

\end{blanko}

\begin{blanko}{General reduced-cover squares.}
  More generally, instead of starting with the reduced cover of $[k]$ consisting 
  of the $k$ maps $\rho_i: [1] \inertto [k]$, we can start with an 
  arbitrary reduced cover of $[k]$, namely $m$ inert maps $\tau_i : [k_i] 
  \inertto [k]$ with $\sum_i k_i = k$ and such that they are jointly 
  surjective and $k_i \neq 0$.
  With this data, just as before, we write $[n_i]$ for the
  ordinal appearing in the active-inert factorisation of $\alpha \circ 
  \tau_i$:
  \[
  \begin{tikzcd}
  {[k_i]} \ar[r, inertto, "\tau_i"] \ar[d, dotted, -act, "\alpha_i"']
  & {[k]} \ar[d, -act, "\alpha"]
  \\
  {[n_i]}  \ar[r, inertto, dotted, "\gamma^{\alpha,\tau}_i"'] &
  {[n]}  .
  \end{tikzcd}
  \]
  Again, if $k>0$, the maps $\gamma^{\alpha,\tau}_i: [n_i] \inertto [n]$
  together constitute a cover of $[n]$, which is reduced if $\alpha$ is
  injective. Note also that we have $\alpha = \alpha_1 \vee \cdots \vee
  \alpha_m$. For convenience we assume the cover is in the canonical order,
  that is, $\tau_{i}(0)<\tau_{i+1}(0)$ for $1\leq i\leq m-1$, and write
  $\beta:[m]\actto[k]$ for the active map with $\beta(i)=\tau_{i+1}(0)$.

  The squares together define for any simplicial space $X$
  a diagram
  \begin{equation*}\label{alpha-tau-pbk}\tag{GRCS}
  \begin{tikzcd}[column sep={10em,between origins}]
  X_{k_1} \times \cdots \times X_{k_m} & 
  \ar[l, "{(\tau_1,\ldots,\tau_m)\upperstar}"']   X_k   \\
  X_{n_1} \times \cdots \times X_{n_m} 
  \ar[u, -act, "(\alpha_1 \times\cdots\times \alpha_m)\upperstar"] & X_n  .
  \ar[u, -act, "{\alpha\upperstar = (\alpha_1\vee\cdots\vee\alpha_k)\upperstar}"'] 
  \ar[l, "{(\gamma^{\alpha,\tau}_1,\ldots,\gamma^{\alpha,\tau}_m)\upperstar}"]
  \end{tikzcd}
  \end{equation*}
  We refer to these squares as {\em general reduced-cover squares}.
  Note again that the vertical maps are active and the components of the 
  horizontal maps are inert.  
\end{blanko}

\begin{theorem}\label{thm:decomp}
  For any simplicial space $X$, the following are equivalent.
  \begin{enumerate}
    \item Active-inert squares are pullbacks (i.e.~$X$ is a decomposition 
	space).
  
    \item Squares formed by inert maps and active injections are pullbacks.
  
    \item For every active injection $\alpha: [k] \actto [n]$ with $k\neq 
	0$, the special reduced-cover square \eqref{alpha-pbk} is a pullback.
  
    \item For every reduced cover $( \tau_i : [k_i] \inertto [k]) _{1\leq 
	i \leq m}$ and every every active injection $\alpha: [k] \actto [n]$ with $k\neq 
	0$, the general reduced-cover square \eqref{alpha-tau-pbk} is a pullback.
  \end{enumerate}
\end{theorem}
  
\begin{proof}
  The equivalence of (1) and (2) is the content of the theorem of Feller et
  al.~\cite{Feller-Garner-Kock-Proulx-Weber:1905.09580} (that is, the
  statement that every $2$-Segal space is unital).
  
  The special reduced-cover squares \eqref{alpha-pbk} are special cases 
  of the general reduced-cover squares \eqref{alpha-tau-pbk}, so it is clear that (4) implies (3).
  Conversely, (3) implies (4) by an easy prism-lemma argument.
  Write an arbitrary general reduced-cover square as the bottom square:
  \[
  \begin{tikzcd}[column sep={10em,between origins}]
  X_1 \times \cdots \times X_1 & \ar[l, 
  "{(\rho_1,\ldots,\rho_m)\upperstar}"']   X_m   \\
  X_{k_1} \times \cdots \times X_{k_m} \ar[u, -act]& 
  \ar[l, "{(\tau_1,\ldots,\tau_m)\upperstar}"']   X_k  \ar[u, -act, "\beta\upperstar"'] \\
  X_{n_1} \times \cdots \times X_{n_m} 
  \ar[u, -act, "(\alpha_1 \times\cdots\times \alpha_m)\upperstar"] & X_n
  \ar[u, -act, "{\alpha\upperstar}"'] 
  \ar[l, "{(\gamma^{\alpha,\tau}_1,\ldots,\gamma^{\alpha,\tau}_m)\upperstar}"]
  \end{tikzcd}
  \]
  and complete it by pasting a special reduced-cover square on top of it.
  Assuming Condition~(3), both the upper square and the whole rectangle are
  pullbacks, so by the prism lemma also the lower square is a pullback,
  which means that (4) holds.

  It is not difficult to show that (4) implies (2): we want to establish that the square
  \[
  \begin{tikzcd}
  X_n\ar[d, -act, "d_i"'] & \ar[l, inertto, "d_\top"'] X_{n+1} \ar[d, -act, "d_i"]  \\
  X_{n-1} & X_n \ar[l, inertto, "d_\top"]
  \end{tikzcd}
  \]
  is a pullback (for $0 < i <n$). (We should of course similarly deal with 
  the analogous squares with bottom face maps.) Decompose the square as
  \[
  \begin{tikzcd}
  X_n\ar[d, -act, "d_i"'] & \ar[l, "\operatorname{pr}_1"'] X_n \times X_1 \ar[d, -act, "d_i\times 
  \id"']& \ar[l, 
  "{(d_\top,d_\bot^n)}"'] X_{n+1} \ar[d, -act, "d_i"]  \\
  X_{n-1} & \ar[l, "\operatorname{pr}_1"] X_{n-1} \times X_1 & \ar[l, "{(d_\top,d_\bot^n)}"] X_n  .
  \end{tikzcd}
  \]
  Now the left-hand square is a pullback since it projects away an 
  identity, and the right-hand square is a pullback since it is a general 
  reduced-cover square as in Condition~(4)
  
  The most 
  interesting part is to show that (2) implies (4). So we assume that all 
  the squares in Condition~(2) are pullbacks, and aim to show that a general 
  reduced-cover square \eqref{alpha-tau-pbk} is a pullback. 
  For ease of exposition we describe explicitly
  the case where there are only $m=2$ charts in the cover $\tau$.
  This means that the square has the form
  \[
    \hspace*{2.6em}
    \begin{tikzcd}[column sep={9em,between origins}, row sep={5em,between origins}]
{}&  X_{k_1} \times X_{k_2} & \ar[l,
  "{(d_\top^{k_2},d_\bot^{k_1})}"']   X_k   \\
{}&  X_{n_1} \times  X_{n_2} \ar[u, -act, "(\alpha_1\times \alpha_2)\upperstar"] & X_n .
  \ar[u, -act, "{\alpha\upperstar =(\alpha_1 \vee \alpha_2)\upperstar}"']
  \ar[l, "{(d_\top^{n_2},d_\bot^{n_1})}"]
  \end{tikzcd}
  \]
  Such a square we can decompose into two (or $m$, in the general case)
  smaller squares vertically like the solid part of this diagram:
  \[
  \begin{tikzcd}[column sep={9em,between origins}, row sep={2.4em,between origins}]
	X_{k_2}& & 
	\\
 &  X_{k_1} \times X_{k_2} \ar[lu, dotted, "\operatorname{pr}_2"']
 & \ar[l, "{(d_\top^{k_2},d_\bot^{k_1})}"']
  X_{k_1+k_2}   
  \\
X_{n_2} \ar[uu, -act, dotted, "\alpha_2\upperstar"] & & 
 \\
 & X_{k_1} \times  X_{n_2} 
 \ar[ld, dotted, "\operatorname{pr}_1"] \ar[lu, dotted, "\operatorname{pr}_2"']
  \ar[uu, -act, "(\id \times \alpha_2)\upperstar"'] &
  X_{k_1+n_2} 
  \ar[uu, -act, "{(\id \vee \alpha_2)\upperstar}"']
  \ar[l, "{(d_\top^{n_2},d_\bot^{k_1})}"] 
  \\
X_{k_1} & & 
 \\
 & X_{n_1} \times  X_{n_2} \ar[ld, dotted, "\operatorname{pr}_1"]
  \ar[uu, -act, "(\alpha_1\times \id)\upperstar"'] &
  X_{n_1+n_2} .
  \ar[uu, -act, "{(\alpha_1 \vee \id)\upperstar}"']
  \ar[l, "{(d_\top^{n_2},d_\bot^{n_1})}"] \\
X_{n_1} \ar[uu, -act, dotted, "\alpha_1\upperstar"] &&
  \end{tikzcd}
    \]
  The dotted projection squares, which are pullbacks, serve to show that the two
  (respectively $m$)
  solid squares are pullbacks. Indeed, the horizontal rectangles
  are pullbacks because they are active-inert pullbacks (under 
  Condition~(2)), so by the prism lemma the solid squares are pullbacks,
  and therefore the vertical solid rectangle is a pullback, as we wanted 
  to show.
\end{proof}

\begin{remark}
  The equivalences involving Conditions (3) and (4) are new in this
  generality, as far as we know. A version of (1)$\Leftrightarrow$(4) but
  with all active maps instead of only active injections (hence a weaker
  statement) was given in \cite[Prop.~6.9]{Galvez-Kock-Tonks:1512.07573}
  via a detour into the twisted arrow category of the category of active
  maps. The full strength of Condition~(3) is important in the following,
  because it is the one that immediately interacts with the notion of
  semi-ikeo map, which we come to next. (In particular, Condition~(3) is
  the key to Proposition~\ref{semiikeo / decomp is decomp}.)
\end{remark}

\subsection{Convolution and M\"obius function}

A combinatorial coalgebra is generally the vector space spanned
by the iso-classes of certain combinatorial objects (classically
intervals in a given poset), and the comultiplication is given in terms of
decomposition of those objects. Linear functionals on such a coalgebra,
such as the zeta and M\"obius functions, form the convolution algebra.
Homotopy linear algebra~\cite{Galvez-Kock-Tonks:1602.05082} gives a rather
systematic way of lifting such constructions to the objective level and
transforming algebraic proofs into bijective ones. Instead of the vector space
spanned by iso-classes of combinatorial objects, one considers the slice
category over the groupoid (or $\infty$-groupoid) $I$ of the combinatorial
objects themselves, with linear functors between such slices. Linear
functors are given by spans $I \leftarrow M \to J$, and instead of
algebraic identities one looks for homotopy equivalences between spans.

The reason why this works so well is that the slice category over $I$ is
the homotopy-sum completion of $I$, just as a vector space is the
linear-combination completion, and that linear functor means homotopy-sum
preserving, just like linear map means linear-combination preserving.
Furthermore the span representation of a linear functor corresponds to the
matrix representation of a linear map.
Thus the standard algebraic identities can be
recovered from these homotopy equivalences by taking homotopy cardinality,
under certain finiteness conditions. Specifically, all spans must be of
finite type meaning that the left leg $I \leftarrow M$ must have (homotopy)
finite fibres. But it is usually the case that the homotopy equivalences
can be established even without the finiteness conditions.

Let us briefly see how this procedure looks in the case of
interest, M\"obius functions~\cite{Galvez-Kock-Tonks:1512.07577}.
Recall that for any decomposition space $X$, the {\em incidence coalgebra} is the 
$\infty$-category $\spaces_{/X_1}$ equipped with the comultiplication $\Delta$ and
counit $\epsilon$ given by the spans
$$
X_1 \stackrel{d_1}\longleftarrow X_2 \stackrel{(d_2,d_0)}\longrightarrow 
X_1 \times X_1 \qquad \qquad
X_1 \stackrel{s_0}\longleftarrow X_0 \longrightarrow 1 .
$$
The incidence {\em algebra} is the convolution algebra 
$\kat{Lin}(\spaces_{/X_1}, \spaces) \simeq \spaces^{X_1}$.
Its objects are linear functionals, that is, given by spans $X_1
\leftarrow F \to 1$, with the standard convolution product $*$ given by 
the pullback formula
\[
\begin{tikzcd}[column sep={6em,between origins}, row sep={5em,between 
  origins}]
X_1 && \\
X_2 \ar[u, "d_1"] \ar[d, "{(d_2,d_0)}"']  & F*G \dlpullback \ar[lu] \ar[l] \ar[d] \ar[rd] & \\
X_1\times X_1 & F\times G \ar[l] \ar[r] & 1,
\end{tikzcd}
\]
and unit $\epsilon$. The incidence algebra at the level of
$\Q$-vector spaces is obtained by taking homotopy cardinality, provided 
certain finiteness conditions hold,
cf.~Subsection~\ref{sub:fin} below.

\bigskip

The relevance of the `inert' characterisation of decomposition spaces is
that it shows to what extent one can compose. Composition in the sense of
arrows in a category is not possible, but the convolution product provides
an alternative.
  In a Segal space, given a $p$-simplex whose last vertex coincides 
  with the zeroth simplex of a $q$-simplex, one can compose to get an
  $(p+q)$-simplex. This is provided by the equivalence $X_p \times_{X_0} 
  X_q \simeq X_{p+q}$. This is not generally possible in a decomposition 
  space, but it {\em is} possible in case the $p$-simplex and the 
  $q$-simplex already `sit on a $2$-simplex':
  if the long
  edges of the two simplices form the short edges of a $2$-simplex, then
  the $2$-simplex serves as a mould for the gluing.
  
  This is precisely what the convolution product allows, thanks to the
  decomposition-space axiom, which naturally appears in the `inert' form:
  to convolve the linear functionals $X_1 \leftarrow X_p\to 1$ and $X_1
  \leftarrow X_q\to 1$ (where the left-hand maps send a simplex to its long
  edge), we follow the pullback formula above to get
\[
\begin{tikzcd}[column sep={6em,between origins}, row sep={5em,between 
  origins}]
X_1 && \\
X_2 \ar[u, "d_1"] \ar[d, "{(d_2,d_0)}"']  & X_{p+q} \dlpullback \ar[lu] 
\ar[-act,l]
\ar[d] \ar[rd] & \\
X_1\times X_1 & X_p\times X_q
\ar[-act,l]
\ar[r] & 1,
\end{tikzcd}
\]
That $X_{p+q}$ appears as the pullback is precisely one of the basic
instances of the decomposition space axiom, inert version 
(Theorem~\ref{thm:decomp}).
 
\begin{blanko}{Completeness.}
  A decomposition space is called {\em complete}~\cite{Galvez-Kock-Tonks:1512.07577}
  when $s_0 : X_0 \to X_1$ is mono. The complement is then denoted 
  $\nondeg X_1$, the space of {\em nondegenerate edges},
  so as to be able to write $X_1 = X_0 + \nondeg X_1$. Since in a 
  decomposition space all degeneracy maps are pullbacks of this first 
  $s_0$ (cf.~\cite{Galvez-Kock-Tonks:1512.07577}), it follows that they are all mono, 
  and there is a well-defined space $\nondeg X_n \subset X_n$ of 
  nondegenerate $n$-simplices.
\end{blanko}
 
\begin{lemma}[{\cite{Galvez-Kock-Tonks:1512.07577}}]\label{nondeg by edges}
  An $n$-simplex of a complete decomposition space 
  is nondegenerate if and only if each of its principal 
  edges is nondegenerate.
\end{lemma}

\begin{blanko}{Phi functors.}\label{Phi}
  For each $n$, we define $\Phi_n$ to be the linear functional given by the span
  $$
  X_1 \leftarrow \nondeg X_n \to 1 .
  $$
  The left-hand map sends an $n$-simplex to its long edge.
\end{blanko}
 
\begin{remark}
  The $\Phi$-notation goes back to Leroux~\cite{Leroux:1976} (his 
  {\em \'el\'ements remarquables}), and was preserved by Lawvere and 
  Menni~\cite{LawvereMenniMR2720184}.
\end{remark}

The convolution formula $X_p * X_q = X_{p+q}$ from above restricts to
nondegenerate simplices to give the following fundamental 
formula.

\begin{lemma}\label{Phipq}
  For any complete decomposition space we have
  $$
  \Phi_p * \Phi_q = \Phi_{p+q} .
  $$
\end{lemma}

\begin{blanko}{M\"obius function.}\label{mu}
  The importance of the Phi functors is that the M\"obius function can be 
  described as
  $$
  \mu = \Phieven - \Phiodd = \sum_{n\in \N} (-1)^n \Phi_n .
  $$
  More precisely, it is the linear functional $\spaces_{/X_1} \to \spaces$
  given by the span
  $$
  X_1 \leftarrow \sum_{n\in \N} (-1)^n \Phi_n \to 1 .
  $$
  The minus signs does not immediately make sense at the objective level,
  but the equation that the M\"obius function is required to satisfy,
  $$
  \mu * \zeta = \epsilon
  $$
  can be rewritten by spelling out in terms of Phi functors and then moving
  the negative terms to the other side of the equation. The resulting formula
  $$
  \Phieven * \zeta = \epsilon + \Phiodd * \zeta
  $$
  makes sense at the objective level, and it can be established as an
  explicit homotopy equivalence of
  $\infty$-groupoids~\cite{Galvez-Kock-Tonks:1512.07577}.
\end{blanko}

\section{Ikeo and semi-ikeo maps}
\label{sec:ikeo}

A simplicial map $f:Y \to X$ defines a linear map on incidence algebras
$f\lowershriek : \spaces^{Y_1} \to \spaces^{X_1}$ by sending a linear
functional $Y_1 \leftarrow F \to 1$ to the linear functional $X_1
\leftarrow Y_1 \leftarrow F \to 1$. If $f: Y \to X$ is ikeo, then this
linear map will preserve the convolution product and the unit $\epsilon$ so
as to define an algebra map $\spaces^{Y_1} \to \spaces^{X_1}$. In the
situation of this paper, $f$ will not be ikeo, but it will still be {\em
semi-ikeo} (cf.~below). This condition is enough to ensure that
$f\lowershriek$ preserves the convolution product (although it will not
preserve the algebra unit $\epsilon$).

\subsection{Ikeo maps}

A simplicial map $Y\to X$ is called {\em ikeo} when for every active map 
$\alpha: [k] \actto [n]$ the square
\begin{equation}\label{ikeo alpha-instance}
\begin{tikzcd}[column sep={10em,between origins}]
 Y_{n_1} \times \cdots \times Y_{n_k} \ar[d]& Y_n \ar[d]
 \ar[l, "{(\gamma^\alpha_1,\ldots,\gamma^\alpha_k)\upperstar}"'] \\
 X_{n_1} \times \cdots \times X_{n_k} & X_n 
 \ar[l, "{(\gamma^\alpha_1,\ldots,\gamma^\alpha_k)\upperstar}"]
\end{tikzcd}
\end{equation}
is a pullback.

The following two more economical criteria are useful.

\begin{lemma}\label{ikeo n}
  For a general simplicial map $Y \to X$,
  the ikeo condition is equivalent to demanding that for each $n\geq 
  0$ the square 
  \begin{equation}\label{ikeo n-instance}
  \begin{tikzcd}[column sep={10em,between origins}]
   Y_1 \times \cdots \times Y_1 \ar[d]& Y_n \ar[d]
   \ar[l, "{(
      \rho_{1},\ldots,\rho_{n}
     )\upperstar}"'] \\
   X_1 \times \cdots \times X_1 & X_n 
   \ar[l, "{(
       \rho_{1},\ldots,\rho_{n}
     )\upperstar}"]
  \end{tikzcd}
  \end{equation}
  is a pullback.
\end{lemma}

\begin{proof}
  Since square~\eqref{ikeo n-instance} is a special case of 
  square~\eqref{ikeo alpha-instance} where $\alpha$ is the identity map, it is clear that ikeo implies
  the condition of the lemma. Conversely suppose the condition of the lemma is satisfied, and 
  consider a general square, as on the right in this diagram:
  \[
\begin{tikzcd}[column sep=5.5em]
 (Y_1 \!\times\!\cdots\!\times\! Y_1) \times \cdots \times (Y_1 \!\times\!\cdots\!\times\! Y_1) 
 \ar[d] 
 & \ar[l] Y_{n_1} \times \cdots \times Y_{n_k} \ar[d]& Y_n \ar[d]
 \ar[l, "{(\gamma^\alpha_1,\ldots,\gamma^\alpha_k)\upperstar}"'] \\
 (X_1 \!\times\!\cdots\!\times\! X_1) \times \cdots \times (X_1 \!\times\!\cdots\!\times\! X_1) 
 & \ar[l] X_{n_1} \times \cdots \times X_{n_k} & X_n   .
 \ar[l, "{(\gamma^\alpha_1,\ldots,\gamma^\alpha_k)\upperstar}"]
\end{tikzcd}
\]
The outer rectangle is the $n$-instance of square~\eqref{ikeo n-instance}, so 
it is a pullback.
The left-hand square is the product of $k$ squares, which are the 
$n_i$-instances of \eqref{ikeo n-instance}, so it is a pullback too.
Therefore the right-hand-square is a pullback, by the prism lemma.
\end{proof}

\begin{lemma}\label{lem:0+2}
To check that a simplicial map $Y \to X$ is ikeo, it is enough to check 
  it for active maps $[0]\actto [0]$ and $[2] \actto [n]$. In other words, 
  it is enough to check that
  the squares
  \[
  \begin{tikzcd}
 1 \ar[d]& Y_0 \ar[d] \ar[l] \\
 1 & X_0 \ar[l]
\end{tikzcd}
\qquad\quad
\begin{tikzcd}[column sep={6em,between origins}]
 Y_{n_1} \times  Y_{n_2} \ar[d]& Y_n \ar[d]
 \ar[l] \\
 X_{n_1} \times X_{n_2} & X_n 
 \ar[l]
\end{tikzcd}
\]
are pullbacks for all $n=n_1+n_2$.
\end{lemma}
\begin{proof}
  Assuming the indicated pullback squares for $k=0$ and $k=2$, we need to consider the corresponding square for a general active map
  $\alpha: [k] \actto [n]$.
  For $k=1$ the square is a pullback since its horizontal maps are identities.
  For $k\geq 2$, the square can be decomposed as the
  pasting of squares
  \[
  \begin{tikzcd}[column sep={9em,between origins}]
 Y_{n_1} \times \cdots \times Y_{n_k} \ar[d] & 
 \phantom{xxx}\cdots\phantom{xxx} \ar[l] & \ar[l]
 Y_{n_1} \times Y_{n_2+\cdots+n_k}  \ar[d]& Y_{n_1+\cdots+n_k} \ar[d]
 \ar[l] \\
 X_{n_1} \times \cdots \times X_{n_k} & \phantom{xxx}\cdots\phantom{xxx}  \ar[l] &  X_{n_1} \times X_{n_2+\cdots+n_k} 
 \ar[l]& X_{n_1+\cdots+n_k}
 \ar[l]
\end{tikzcd}
\]
  Here the rightmost square is a ($k\!=\!2$)-instance, and the remaining
  squares to the left are products of
  ($k\!=\!1$)-instances with a ($k\!=\!2$)-instance.
  (The case $k\!=\!0$
  is not covered by this argument, which is why it has to be listed 
  separately in the lemma.)
\end{proof}

Note that the identity map $[0] \actto [0]$ gives the square
\begin{equation}\label{eq:0-square}
\begin{tikzcd}
 1 \ar[d]& Y_0 \ar[d] \ar[l] \\
 1 & X_0 \ar[l]
\end{tikzcd}
\end{equation}
which is a pullback if and only if $Y_0 \to X_0$ is an
equivalence, that is,
if the simplicial map is an `equivalence on objects'.

Note also that the identity map $[2] \actto [2]$ gives the square
\begin{equation}\label{eq:2-square}
\begin{tikzcd}
 Y_1\times Y_1 \ar[d]& Y_2 \ar[d] \ar[l, "{(d_2,d_0)}"'] \\
 X_1 \times X_1 & X_2 . \ar[l, "{(d_2,d_0)}"]
\end{tikzcd}
\end{equation}

These two squares are common to both the previous lemmas, and in fact we
have:

\begin{lemma}\label{lem:ikeo2}
If $X$ and $Y$ are decomposition spaces, then to check that a simplicial map $Y\to X$ 
  is ikeo, it is enough to check the two squares \eqref{eq:0-square} and 
  \eqref{eq:2-square}.
\end{lemma}

\begin{proof}
  By Lemma~\ref{lem:0+2} it is enough to establish for each 
  $\alpha:[2]\actto [n]$ (the join of $\alpha_1:[1] \actto [n_1]$ and 
  $\alpha_2:[1] \actto [n_2]$) that the following 
  back face is a pullback:
  \[
  \begin{tikzcd}[column sep={6em,between origins}]
 Y_{n_1} \times  Y_{n_2}
 \ar[d]& Y_n \ar[rrd, pos=0.2, "{\!\alpha\upperstar}"] \ar[d]
 \ar[l] \\
 X_{n_1} \times X_{n_2} \ar[rrd, pos=0.3, 
 "{(\alpha_1\times\alpha_2)\upperstar\!\!}"']  & X_n \ar[rrd, pos=0.2, "{\!\alpha\upperstar\!\!}"] \ar[l] & Y_1 \times  Y_1& Y_2 \ar[l]
 \ar[d] \\
  & & X_1 \times X_1 & X_2  \ar[l]
  \ar[l]
  \arrow[from=1-1, to=2-3, crossing over, pos=0.3,
 "{(\alpha_1\times\alpha_2)\upperstar}"']
  \arrow[from=2-3, to=3-3, crossing over]
\end{tikzcd}
\]
  But this follows by a prism-lemma argument 
  from the fact that the front face is a pullback by 
  assumption. Indeed, the top and bottom faces are pullbacks since $Y$ and 
  $X$ are decomposition spaces (by Condition (3) in
  Theorem~\ref{thm:decomp}).
\end{proof}

The word {\em ikeo} is an acronym standing for `inner Kan and equivalence 
on objects', but these two notions
have a meaning individually, and it is actually a lemma that the notions 
match up.

Recall that a simplicial map is called {\em inner Kan} (or {\em relatively Segal})
if for each $n\geq 2$ the square
\begin{equation}\label{inner-kan-n-instance}
\begin{tikzcd}
Y_1 \times_{Y_0} \cdots \times_{Y_0} Y_1 \ar[d] & Y_n \ar[d] \ar[l] \\
X_1 \times_{X_0}\cdots \times_{X_0} X_1 & X_n \ar[l]
\end{tikzcd}
\end{equation}
is a pullback.

\bigskip

Note that if both $X$ and $Y$ are Segal spaces, then every simplicial map 
$Y \to X$ is relatively Segal.

\begin{lemma}
  A map is ikeo if and only if it is inner Kan and an equivalence on 
  objects.
\end{lemma}
\begin{proof}
  The $n=0$ case of the ikeo condition says that the map is an equivalence
  on objects. We show that the $n=2$ instance of \eqref{ikeo n-instance} is
  a pullback if and only if the $n=2$ instance of
  \eqref{inner-kan-n-instance} is a pullback, and leave the rest to the
  reader.
  In the prism  diagram
\[
\begin{tikzcd}
Y_1 \times  Y_1 \ar[d] &Y_1 \times_{Y_0} Y_1 \ar[l] \ar[d] & Y_2 \ar[d] \ar[l] \\
X_1 \times X_1 &X_1 \times_{X_0} X_1 \ar[l] & X_2 \ar[l]
\end{tikzcd}
\]
the left-hand square is a pullback because $Y_0 \to X_0$ is mono, by
Lemma~\ref{prod and fib prod} below. By the prism lemma the right-hand
square is a pullback if and only if the outer rectangle is a pullback.
\end{proof}

In fact the key argument in the proof gives more generally:
\begin{lemma}\label{ikeo<=>inner Kan n instance}
  Let $Y \to X$ be a simplicial map such that $Y_0 \to X_0$ is mono,
  then 
  \[
  \begin{tikzcd}
  Y_1 \times_{Y_0}  Y_1 \ar[d] & Y_2 \dlpullback \ar[d] \ar[l] \\
  X_1 \times_{X_0} X_1 & X_2 \ar[l]
  \end{tikzcd} \qquad \Leftrightarrow \qquad
  \begin{tikzcd}
  Y_1 \times  Y_1 \ar[d] & Y_2 \dlpullback \ar[d] \ar[l] \\
  X_1 \times X_1 & X_2  , 
  \ar[l]
  \end{tikzcd}
  \]
  and similarly for all $n\geq 2$.
\end{lemma}

\subsection{Semi-ikeo maps}

The importance of ikeo maps is that they induce algebra homomorphisms at
the level of incidence algebras.
In our situation we will not have ikeo maps but only
something weaker, where the convolution product is preserved but
the convolution unit $\epsilon$ is not.

Provisionally we call a
simplicial map $f: Y \to X$ {\em semi-ikeo} when for every
active injection $\alpha: [k] \actto [n]$ between nonzero
ordinals, the square~\eqref{ikeo alpha-instance} is a pullback.

Observe that there are semi-ikeo versions of Lemma~\ref{ikeo n},
characterising semi-ikeo maps in terms of $n\geq 1$,
of Lemma~\ref{lem:0+2}, referring only to active injections $[2]\actto [n]$,
and of Lemma~\ref{lem:ikeo2}, saying that if both $X$ and $Y$ are
already known to be decomposition spaces then the semi-ikeo
condition can be checked on the single square
\[
\begin{tikzcd}
  Y_1\times Y_1 \ar[d]& Y_2 \ar[d] \ar[l, "{(d_2,d_0)}"'] \\
  X_1 \times X_1 & X_2 . \ar[l, "{(d_2,d_0)}"]
\end{tikzcd}
\]

\begin{prop}\label{semiikeo / decomp is decomp}
  Given a semi-ikeo simplicial map between simplicial spaces $Y \to X$, if 
  $X$ is a decomposition space, then also $Y$ is a decomposition space.
\end{prop}

\begin{proof}
  By Theorem~\ref{thm:decomp},
  it is enough to establish that the special reduced-cover square
    \[
  \begin{tikzcd}[column sep={10em,between origins}]
  Y_1 \times \cdots \times Y_1 & \ar[l]   Y_k   \\
  Y_{n_1} \times \cdots \times Y_{n_k} \ar[u] & Y_n \ar[u, 
  "{\alpha\upperstar}"'] \ar[l, 
  "{(\gamma^\alpha_1,\ldots,\gamma^\alpha_k)\upperstar}"]
  \end{tikzcd}
  \]
  is a pullback for every active injection $\alpha : [k] \actto [n]$ with 
  $k\neq 0$. We have
      \[
  \begin{tikzcd}[column sep={9em,between origins}]
  X_1 \times \cdots \times X_1 & \ar[l]   X_k   \\
  Y_1 \times \cdots \times Y_1 \ar[u] & \ar[l]   Y_k \ar[u]   \\
  Y_{n_1} \times \cdots \times Y_{n_k} \ar[u] & Y_n \ar[u, 
  "{\alpha\upperstar}"'] \ar[l, 
  "{(\gamma^\alpha_1,\ldots,\gamma^\alpha_k)\upperstar}"]
  \end{tikzcd}
  \qquad
  =
  \qquad
  \begin{tikzcd}[column sep={9em,between origins}]
  X_1 \times \cdots \times X_1 & \ar[l]   X_k   \\
  X_{n_1} \times \cdots \times X_{n_k} \ar[u] &   \ar[l, "{(\gamma^\alpha_1,\ldots,\gamma^\alpha_k)\upperstar}"]
   X_n \ar[u, "{\alpha\upperstar}"']   \\
  Y_{n_1} \times \cdots \times Y_{n_k} \ar[u] & Y_n .
  \ar[u] \ar[l, "{(\gamma^\alpha_1,\ldots,\gamma^\alpha_k)\upperstar}"]
  \end{tikzcd}
  \]
  On the right, the top square is a pullback since $X$ is a decomposition 
  space, and the bottom square is a pullback since $Y\to X$ is semi-ikeo.
  So the outer rectangle (either on the left or on the right) is a 
  pullback. But on the left, the top square is a pullback since $Y \to X$ 
  is semi-ikeo. So it follows from the prism lemma that also the bottom 
  square is a pullback, which is what we needed to prove.
\end{proof}

Note also that Lemma \ref{ikeo<=>inner Kan n instance} actually establishes
the following result.

\begin{lemma}\label{if mono, semiikeo<=>ikeo}
  If $Y \to X$ is mono on objects, then semi-ikeo is equivalent to
  relatively Segal.
\end{lemma}

\begin{remark}
  Without the mono condition, it is not true that relatively Segal implies 
  semi-ikeo. For example, any simplicial map between Segal spaces is 
  relatively Segal. Now take a map from a Segal space $Y$ to the terminal
  simplicial set, then the semi-ikeo condition says that $Y_1 \times Y_1 
  \leftarrow Y_2$ is an equivalence, or equivalently $Y_1\times Y_1 
  \leftarrow Y_1 \times_{Y_0} Y_1$ is an equivalence. Of course this is 
  not generally true (but is clearly true if $Y_0=1$).
\end{remark}

\begin{example}
  A morphism of posets  $f:Y \to X$ is ikeo if and only if it is a 
  bijection on objects. (It does not have to be an isomorphism: for 
  example, $Y$ could be the discrete poset of objects of $Y$.) To be 
  semi-ikeo, it is enough to be injective on objects.
\end{example}

%

\section{Full inclusions and convexity}

\subsection{A few standard facts about monomorphisms of spaces}

Recall that a map of spaces $f:T \to S$ is called a {\em monomorphism} (or just {\em mono}, 
for short) when it is $(-1)$-truncated. That is, its fibres are $(-1)$-truncated, meaning they are each either
contractible or empty. We denote monomorphisms by $\into$.
Alternatively, $f$ is a mono when
\[
\begin{tikzcd}
T \ar[r, "="] \ar[d, "="'] & T \ar[d, "f"]  \\
T \ar[r, "f"'] & S
\end{tikzcd}
\]
is a pullback. 
This last characterisation is just a reformulation of the standard fact that
\begin{lemma}\label{mono if diag}
  $f:T \to S$ is mono if and only if the diagonal map $T \to T \times_S 
T$ is an equivalence.
\end{lemma}
This in turn is a special case of the general fact that a map
$f:T \to S$ is $n$-truncated if and only if its 
diagonal map $T \to T\times_S T$ is $(n-1)$-truncated (see 
Lurie~\cite[5.5.6.15]{Lurie:HTT}).

\begin{lemma}\label{lem:diag pbk}
  A map of spaces $f:T \to S$ is mono if and only if the square
  \[
  \begin{tikzcd}
  T \times T \ar[d, "f\times f"'] & T \ar[l, "\text{\rm diag}"'] \ar[d, "f"]  \\
  S \times S & S \ar[l, "\text{\rm diag}"]
  \end{tikzcd}
  \]
  is a pullback.
\end{lemma}
\begin{proof}
The square is a pullback if and only if, for each $s\in S$, the induced map on fibres
$$
(f\times f)^{-1}(s,s) \leftarrow f^{-1}(s)
$$
is an equivalence. But this map is the diagonal of $f^{-1}(s)\to 1$, so it
is an equivalence if and only if $f^{-1}(s)\to 1$ is mono, by
Lemma~\ref{mono if diag}. This condition for each $s\in S$ is the condition
for $f$ to be mono.
\end{proof}

The following easy lemma is standard; we state it since it is used several 
times. (We include the proof because it is pleasant.)
\begin{lemma}\label{pbk-change}
  In the situation
  \[
  \begin{tikzcd}
   & Y \ar[d] & \\
  X \ar[r] & T \ar[rd, into, "f"] & \\
  && S ,
  \end{tikzcd}
  \]
  when $f$ is mono, then the canonical map $X \times_T Y \to X\times_S Y$ is an equivalence.
\end{lemma}

\begin{proof}
  In the diagram
  \[
  \begin{tikzcd}
  P \drpullback \ar[d] \ar[r] & Y \drpullback\ar[d] \ar[r, "="] & Y \ar[d]  \\
  X \drpullback\ar[d, "="'] \ar[r] & T \drpullback\ar[d, "="'] \ar[r, "="] 
  & T \ar[d, into, "f"] \\
  X \ar[r] & T \ar[r, into, "f"'] & S
  \end{tikzcd}
  \]
  we see that $P$ is both the pullbacks.
\end{proof}

\begin{lemma}\label{prod and fib prod}
  If a simplicial map $Y \to X$ is mono on objects, then the square
  \[
  \begin{tikzcd}
  Y_1 \times Y_1 \ar[d] & Y_1 \times_{Y_0} Y_1 \ar[d] \ar[l]  \\
  X_1 \times X_1 & X_1 \times_{X_0} X_1 \ar[l]
  \end{tikzcd}
  \]
  is a pullback.
\end{lemma}

\begin{proof}
  We can use $Y_1 \times_{X_0} Y_1$ instead of $Y_1 \times_{Y_0} Y_1$,
  by Lemma~\ref{pbk-change}.
  Now write the prism
    \[
  \begin{tikzcd}
  Y_1 \times Y_1 \ar[d] & Y_1 \times_{X_0} Y_1  \ar[d] \ar[l]  \\
  X_1 \times X_1 \ar[d] & X_1 \times_{X_0} X_1 \ar[l] \ar[d]  \\
  X_0 \times X_0 & X_0   .
  \ar[l, "\text{diag}"]
  \end{tikzcd}
  \]
  Here both the bottom square and the outer rectangle are pullbacks,
  so it follows (by the prism lemma) that the top square is a pullback.
\end{proof}

\subsection{Full inclusions}

A simplicial map $Y \to X$ is called {\em fully faithful} when for each $n\geq 
0$ the square
\[
\begin{tikzcd}
Y_0 \times \cdots \times Y_0 \ar[d] & Y_n \ar[l] \ar[d]  \\
X_0 \times \cdots \times X_0 & X_n \ar[l]
\end{tikzcd}
\]
is a pullback. The horizontal maps send an $n$-simplex to the 
$(n+1)$-tuple of vertices.

Note that in the case where $X$ and $Y$ are Segal spaces, this condition 
is equivalent to the $n=2$ case, so for Segal spaces the definition agrees
with the usual definition of fully faithful.

\bigskip

A {\em full inclusion} of simplicial sets is by definition a fully faithful 
simplicial map which is furthermore a monomorphism in simplicial degree 
$0$.

\bigskip

Recall (from~\cite{Galvez-Kock-Tonks:1512.07577}) that a simplicial map is 
called {\em conservative} if it is cartesian on all degeneracy maps. (Note that for 
simplicial maps between decomposition spaces, this can be measured on
the first degeneracy map $s_0 : X_0 \to X_1$ alone.)

\begin{lemma}
  A full inclusion is conservative.
\end{lemma}

\begin{proof}
  Let $f: Y \to X$ be a full inclusion. In the cube diagram (for $0 \leq i 
  < n$)
  \[
  \begin{tikzcd}
    & Y_n \ar[ldd]\ar[d] & Y_{n-1} 
    \ar[d]\ar[l, pos=0.7, "s_i"']
    \\
  & X_n \ar[ldd] & X_{n-1} \ar[ldd] \ar[l,pos=0.7, "s_i"'] \\
  Y_0 \cdots Y_0 \ar[d] & Y_0 \cdots Y_0 \ar[d] \ar[l,crossing over] & \\
  X_0 \cdots X_0 & X_0 \cdots X_0 \ar[l] &
\arrow[from=1-3,to=3-2,crossing over]
\end{tikzcd}
  \]
  the sides are pullbacks since $f$ is fully faithful. In the front square,
  there are $n+1$ factors on the left and $n$ factors on the right, and the
  horizontal maps are given by a diagonal in position $i$. So this square 
  is a pullback by Lemma~\ref{lem:diag pbk} since $f$ is mono on 
  objects. Therefore by the prism lemma, the back square is a 
  pullback, and since this holds for all $0 \leq i 
  < n$, this is precisely to say that $f$ is conservative.
\end{proof}

\begin{cor}\label{full=>complete}
  If $f:Y\to X$ is a full inclusion and if $X$ is complete, then also $Y$ is 
  complete.
\end{cor}

\begin{proof}
  This follows since clearly conservative over complete is complete.
\end{proof}

\begin{prop}\label{ff=>ik}
   A full inclusion $f:Y \to X$ is relatively Segal (inner Kan), and so
  semi-ikeo.
\end{prop}

\begin{proof}
  We do the $n=2$ case. We need to show that the square
  \[
  \begin{tikzcd}
  Y_1 \times_{Y_0}  Y_1 \ar[d] & Y_2 \ar[d] \ar[l] \\
  X_1 \times_{X_0} X_1 & X_2 \ar[l]
  \end{tikzcd}
  \]
  is a pullback.
  Consider the prism diagram
  \[
  \begin{tikzcd}
  Y_0 \times Y_0 \times Y_0 \ar[d] \ar[r, phantom, "=" description] &  \ar[d] (Y_0 \times Y_0) \times_{Y_0} (Y_0 \times Y_0) & \ar[l]
  Y_1 \times_{Y_0}  Y_1 \ar[d] & Y_2 \ar[d] \ar[l] \\
  X_0 \times X_0 \times X_0 \ar[r, phantom, "=" description] & (X_0 \times X_0) \times_{X_0} (X_0 \times X_0) & \ar[l]
  X_1 \times_{X_0} X_1 & X_2  .
  \ar[l]
  \end{tikzcd}
  \]
  The middle square is a pullback since it is the fibre
  product over $X_0$ of two copies of the pullback square
  \[
  \begin{tikzcd}
  Y_0 \times Y_0 \ar[d] & Y_1 \ar[l] \ar[d] \dlpullback  \\
  X_0 \times X_0 & X_1 \ar[l]
  \end{tikzcd}
  \]
  expressing that $f$ is fully faithful. Note that this is where we use
  that $Y_0 \to X_0$ is mono, so that pullbacks over $Y_0$ can be computed
  over $X_0$ (cf.~Lemma~\ref{pbk-change}). The outer rectangle is a
  pullback since $Y \to X$ is fully faithful. The prism lemma now tells us
  that the right-hand square is a pullback.
\end{proof}

\begin{blanko}{Full hull.}
  Generally, for a subset $T$ of points of $X_0$, let $Y_0$ denote the 
  full subspace of $X_0$ spanned by $T$, to 
  get a monomorphism $Y_0 \to X_0$. Now consider all simplices of $X$ that have 
  vertices in $Y_0$. Formally define $Y_n$ to be the pullback
  \[
  \begin{tikzcd}
  Y_0 \times \cdots \times Y_0 \ar[d] & Y_{n} \ar[l] \ar[d] \dlpullback  \\
  X_0 \times \cdots \times X_0 & X_n .
  \ar[l]
  \end{tikzcd}
  \]
  The $Y_n$ assemble into a simplicial space, where the face and 
  degeneracy maps are induced from those of $X$.
\end{blanko}

\subsection{Convexity}

A simplicial map $Y \to X$ is called {\em convex} if it is a full 
inclusion which is also culf.

Recall that culf means cartesian on active maps. For decomposition spaces,
this can be measured on the single square
\[
\begin{tikzcd}
Y_1  \ar[d] & Y_2 \ar[l, "d_1"'] \ar[d] \dlpullback  \\
X_1 & X_2 .
\ar[l, "d_1"] 
\end{tikzcd}
\]

\begin{blanko}{Non-example.}
  The full inclusion of simplicial spaces $\Delta^{\{0,2\}} \subset \Delta^2$ is not
  convex, as it is not culf.
\end{blanko}

\begin{blanko}{Non-example.}
  The inclusion of simplicial spaces $\{0,1\} \subset \Delta^1$ is culf but
  not convex as it is not full.
\end{blanko}

\begin{blanko}{Convex hull.}
  Let $X$ be a decomposition space. Any collection of points (subset
  $S\subset \pi_0 X_0$) defines a unique convex hull $Y\subset X$. To form
  it, first consider
  all simplices whose zeroth and last vertex belong to
  $S$, and add all their vertices to the collection.
  This gives us $\overline
  S$. Now take the full hull of $\overline S$. This defines a
  simplicial space $Y$, and we claim it is convex in $X$.

It is thanks to the decomposition-space axiom that the
convex-hull construction stabilises
after one step: if we start with 
points $x$ and $z$, and a new point $y$ is introduced between them, then 
one could ask if there is a new simplex from $x$ to $y$ which will then 
introduce further points between $x$ and $y$. This does not happen because
these points would have been introduced already in the first step: indeed,
if there is a simplex from $x$ to $y$, and since $x$ and $y$ already form 
the short edge of a simplex in $X$, there is also a simplex obtained by
gluing these two simplices.  So anything between $x$ and $y$ will have 
been introduced already in the first step.
\end{blanko}

\begin{lemma}
  Suppose $K \subset X$ is convex. If $\sigma\in X_n$ is an $n$-simplex
  whose last vertex belongs to $K$, then there is a unique index $0\leq j 
  \leq n$ such that vertex $j$ belongs to $K$,
  every face after $j$ belongs to $K$ and no face before $j$ belongs to $K$.
\end{lemma}

\begin{proof}
  Denote by $x_{0},x_{1},\dots,x_{n}$ the vertices of $\sigma$.
  Let $j$ be minimal such that $x_j$ belongs to $K$.
  Since both $x_j$ and $x_n$ belong to $K$, it follows from fullness that 
  the $1$-simplex $x_j x_n$ belongs to $K$. But $x_j x_n$ is the long 
  edge of an $(n-j)$-simplex, and this whole $(n-j)$-simplex must belong 
  to $K$ since the inclusion is culf. By minimality of $j$, no earlier
  faces can belong to $K$.
\end{proof}

\section{Crapo complementation formula}

		\newcommand{\conv}{\!*\!}

Let $X$ be a complete decomposition space, and let $K \subset X$ be a
convex subspace. In particular, the full inclusion map $f: K \to X$ is
semi-ikeo (by Proposition~\ref{ff=>ik}), and therefore $K$ is again a
complete decomposition space. Observe that the complement $X \shortsetminus
K$ is the full hull on the complement $X_0 \shortsetminus K_0$. So the
inclusion map $g: X \shortsetminus K \to X$ is also semi-ikeo, so the
complement $X \shortsetminus K$ is again a decomposition space, by
Proposition~\ref{semiikeo / decomp is decomp}, and it is complete by
Corollary~\ref{full=>complete}. With these arguments, we have everything
prepared, and the symbols in the following all make sense. (Culfness of the
inclusion $K \to X$ is not required for the statement, but it will be
crucial for the proof.)

Recall from \ref{mu} that the M\"obius function is defined as
the formal difference $\mu =\Phieven - \Phiodd$, and that its defining 
equation
$$
\mu*\zeta = \epsilon = \zeta*\mu 
$$
should be interpreted as
$$
\Phieven * \zeta = \epsilon + \Phiodd * \zeta
$$
(for the left-hand equation), which is now an explicit
homotopy equivalence (established in 
\cite{Galvez-Kock-Tonks:1512.07577}).

\subsection{Symbolic version}

In this subsection we state the Crapo formula in symbolic form, meaning 
that we employ the symbol $\mu$ for the M\"obius function.
This is shorthand for something that is not exactly a linear functional but only
a formal difference of linear functionals, and therefore it does not directly
have an objective meaning. To interpret it, we should first expand each 
$\mu$-symbol in terms 
of $\Phi$-symbols, and then move all negative terms to the other side of 
the
equation. Once this is done we have an equation that we can aspire to 
establish as an explicit homotopy equivalence. This expansion procedure
is a bit cumbersome,
but quite routine. Once we have the objective statement, involving 
various 
$\Phieven$- and $\Phiodd$-functionals, we can further break it down to an
equivalence involving only individual $\Phi_n$-functionals. These we finally
establish as explicit homotopy equivalences; the global ones are then 
obtained by summing over all $n$ in a suitable way.

We shall do all that in the next subsection, but it is enlightening first 
to see the symbolic proof.
Here is the formula, symbolic version:

\begin{theorem}\label{crapo-symbolic}
  For $K \subset X$ convex we have
  $$
  \mu^X = \mu^{X\shortsetminus K} + \mu^X *\zeta^K *\mu^X .
  $$
\end{theorem}

First of all, the equation takes place for linear functionals on $X$. When
we write $\zeta^K$, we mean $f\lowershriek (\zeta^K)$. Here it should be
stressed that since $f:K\to X$ is semi-ikeo, $f\lowershriek$ preserves the
convolution product $*$, but is not unital since $f$ is not an equivalence
on objects. One of the ingredients in the proof is deduced from a formula
in $K$, so this is where we need that $f\lowershriek$ preserves $*$.

Intuitively, the formula says that the nondegenerate simplices in $X$ are 
either nondegenerate simplices 
that do not meet $K$ (the summand $\mu_{X\shortsetminus K}$) or
they are nondegenerate simplices that meet $K$  --- that is the summand
$\mu^X *\zeta^K *\mu^X$, which is less obvious.

We can derive the formula from four auxiliary propositions, which we list next. Each 
of these propositions will be proved (in Subsection~\ref{sub:lemmas}) by expanding the $\mu$ symbols into $\Phi$
symbols, and then sorting by sign. The `nicknames' listed for these propositions 
serve to stress 
the correspondence between them and the lemmas of the next 
subsection:
there will be, in each case, a
lemma explaining the homotopy equivalence for a fixed $\Phi_n$.

\bigskip

First we have a proposition only about $K$ (not about $X$):

\begin{prop}[The K-proposition]\label{K-prop}
  
  $$
  \mu^K = \mu^K * \zeta^K * \mu^K  .
  $$
\end{prop}

All the following lemmas amount to analysing how a simplex of $X$ lies with respect to $K$. 
For example, the following `meet proposition' says that if a simplex of $X$ has a vertex in $K$,
then by convexity a whole middle part of the simplex must lie in $K$, and
altogether the simplex must be composed of three parts: a first part 
with edges outside (before) $K$, then a middle part wholly inside $K$, and 
finally a part with edges outside (after) $K$. The convolutions 
are the formal expression of these descriptions.

Define $\mu^{\notin K}$ to be the space of nondegenerate 
$n$-simplices of $X$ for any $n\geq 0$ (with sign $(-1)^n$)
such that no edges belong to $K$. (Note that a vertex is allowed to belong 
to $K$.)
Define $\mu^{\cap K}$ to be the space of nondegenerate 
$n$-simplices of $X$ for any $n\geq 0$ such that at least one vertex belongs to 
$K$.
\begin{prop}[The meet proposition]\label{meet-prop}
  $$
  \mu^{\cap K} = \mu^{\notin K} * \mu^K * \mu^{\notin K}  .
  $$
\end{prop}

\begin{prop}[The S-proposition]\label{S-prop}
  $$
  \mu^{\notin K} * \mu^K = \mu * \Phi_0^K  .
  $$
\end{prop}
\begin{prop}[The T-proposition]\label{T-prop}
  $$
  \mu^K * \mu^{\notin K} = \Phi_0^K * \mu  .
  $$
\end{prop}

Note here that $\Phi_0^K$ is the convolution unit for the decomposition 
space $K$, but since $f$ is does not 
preserve the unit ($f$ is not an equivalence on objects, semi-ikeo, not ikeo)
the pushforwarded linear functional $f\lowershriek (\Phi_0^K)$ is not the
convolution unit in $X$. It is the linear functional
$$
X_1 \leftarrow K_0 \to 1  .
$$
Convolving with it from the right (resp.~from the left) has the effect of
imposing the condition that the last (resp.~the zeroth) vertex is in $K$.

\begin{proof}[Proof of Theorem~\ref{crapo-symbolic} using
  Propositions~\ref{K-prop}--\ref{T-prop}.]
  We clearly have
  $$
  \mu = \mu^{X \shortsetminus K} + \mu^{\cap K}:
  $$
  in terms of nondegenerate simplices, either it does or it doesn't have a
  vertex in $K$. Now apply Proposition~\ref{meet-prop} (the meet prop) to
  get
  $$
  = \mu^{X \shortsetminus K} + \mu^{\notin K} * \mu^K * \mu^{\notin K}.
  $$
  Now apply Proposition~\ref{K-prop} (the K-proposition) to the middle factor
  $\mu^K$ to get
  $$
  = \mu^{X \shortsetminus K} + \mu^{\notin K} * \mu^K * \zeta^K * \mu^K * 
  \mu^{\notin K} .
  $$  
  Now apply Proposition~\ref{S-prop} and Proposition~\ref{T-prop} to get
  $$
  \mu^{X \shortsetminus K} + \mu * \zeta^K * \mu
  $$
  (where we suppressed two instances of $\Phi_0^K$, since they are next to
  $\zeta^K$ anyway, and within $K$, the linear functional $\Phi_0^K$ is the
  neutral element for convolution).
\end{proof}

\subsection{Explicit homotopy equivalences}
\label{sub:lemmas}

Here are the individual pieces.

\begin{lemma}[The K-lemma]\label{K-lemma}
  For any complete decomposition space $K$, and for each $m\geq 0$, we have
  $$
  \Phi_m + \sum_{j=0}^{m-1} \Phi_j * \Phi_1 * \Phi_{m-(j+1)} \ = \
  \sum_{k=0}^m \Phi_k * \Phi_{m-k}  .
  $$
\end{lemma}
\begin{proof}
  There are $m+1$ terms on each side, and they match up precisely, 
  once we identify $\Phi_j * \Phi_1 = \Phi_{j+1}$. In detail, the separate 
  term 
  $\Phi_m$ on the LHS is the $k=0$ term on the RHS, $\Phi_0 * \Phi_m$;
  the remaining terms on the LHS correspond to the terms on the RHS 
  by sending the $j$th 
  term to the term indexed by $k:=j+1$: indeed
  $\Phi_j * \Phi_1 * \Phi_{m-(j+1)} \simeq \Phi_{j+1} * \Phi_{m-(j+1)}
  = \Phi_k * \Phi_{m-k}$ by Lemma~\ref{Phipq}.
\end{proof}

\begin{blanko}{Variant.}\label{variant}
  There is another equivalence, where $\Phi_m$ on the LHS is matched with 
  the last summand on the RHS instead of the zeroth.
\end{blanko}

\begin{cor}\label{eqaddedupevenandodd in thm}
  For any complete decomposition space $K$, we have
  $$
  \Phieven + \Phieven\conv\Phi_1\conv\Phiodd + \Phiodd\conv\Phi_1\conv\Phieven = 
  \Phieven\conv\Phi_0\conv\Phieven + \Phiodd\conv\Phi_0\conv\Phiodd  ,
  $$
  and
  $$
  \Phiodd + \Phieven\conv\Phi_1\conv\Phieven + \Phiodd\conv\Phi_1\conv\Phiodd = 
  \Phieven\conv\Phi_0\conv\Phiodd + \Phiodd\conv\Phi_0\conv\Phieven  .
  $$
  
\end{cor}
\begin{proof}
  This is just to add up instances of Lemma~\ref{K-lemma} for all $m$ even and for all $m$ odd.
\end{proof}

In Proposition~\ref{K-prop} we stated the following:
  $$
  \mu^K = \mu^K * \zeta^K * \mu^K  .
  $$
  This is shorthand for an explicit homotopy equivalence of 
  $\infty$-groupoids. To expand, use first $\mu= \Phieven-\Phiodd$:
  $$
  \Phieven - \Phiodd = (\Phieven - \Phiodd) * (\Phi_0+\Phi_1) * (\Phieven 
  - \Phiodd) ,
  $$
  and then expand and move all minus signs to the other side of the equation to
  obtain finally the sign-free meaning of the proposition:
  $$
  \begin{array}{rr}
  \Phieven\! &+ \ \Phieven \conv \Phi_1 \conv \Phiodd + 
  \Phiodd\conv\Phi_1\conv\Phieven \\[8pt]
  &+ \
  \Phieven\conv\Phi_0\conv\Phiodd + \Phiodd\conv\Phi_0\conv\Phieven
  \end{array}
  \ \
  =
  \ \
  \begin{array}{rr}
  &+ \ \Phieven\conv\Phi_0\conv\Phieven + \Phiodd\conv\Phi_0\conv\Phiodd
   \\[8pt]
  \Phiodd\! &+ \
  \Phieven\conv\Phi_1\conv\Phieven + \Phiodd\conv\Phi_1\conv\Phiodd  .
  \end{array}
  $$
  This is the explicit homotopy equivalence we establish. The equation
  has been arranged so that the first line of the equation is the 
  even equation in Corollary~\ref{eqaddedupevenandodd in thm} and the second line is the
  odd equation in Corollary~\ref{eqaddedupevenandodd in thm}.
  This is the objective proof of Proposition~\ref{K-prop}.

\bigskip

Let now $f: K \to X$ be a convex inclusion. All linear functionals pertaining to
$K$ are decorated with a superscript $K$ (such as in $\zeta^K$, $\mu^K$, 
$\Phi_n^K$), but we use those symbols also for their pushforth along 
$f$, so that the symbols occurring really stand for $f\lowershriek(\zeta^K)$, 
$f\lowershriek(\mu^K)$, $f\lowershriek(\Phi_n^K)$, and so on.
By multiplicativity of $f\lowershriek$ (the fact that $f$ is semi-ikeo), 
the equation for $K$ of 
Lemma~\ref{K-lemma} holds also in $X$. We use the same convention for 
the full inclusion $g: X\shortsetminus K \to X$. 

Finally we shall use two more decorations, ${\notin}K$ and ${\cap}K$, such
as in $\Phi_n^{\cap K}$ and $\Phi_n^{\notin K}$. These linear functionals
on $X$ are not a pushforth, and the symbols will be defined formally along
the way.

\bigskip

Define $\nondeg X^{\notin K}_r$ to be the space of nondegenerate $r$-simplices 
of $X$
such that no edges belong to $K$. (Note that a vertex is allowed to belong 
to $K$.)
Formally, this is defined as a pullback:
\[
\begin{tikzcd}
\nondeg X_r^{\notin K} \drpullback \ar[d] \ar[r] & (\nondeg X_1\shortsetminus 
\nondeg K_1) \times \cdots \times (\nondeg X_1\shortsetminus 
\nondeg K_1) \ar[d]  \\
\nondeg X_r \ar[r] & \nondeg X_1 \times \cdots \times \nondeg X_1  .
\end{tikzcd}
\]
(On the right-hand side there are $r$ factors.)
Now  $\Phi^{\notin K}_r$ is defined to be the linear functional given
by the span
$$
X_1 \leftarrow \nondeg X_r^{\notin K} \to 1  .
$$
(Note that the $r=0$ case is $\Phi_0^{\notin K} = \Phi_0$.)

\bigskip

Denote by $\Phi^{\cap K}_n$ the space of nondegenerate $n$-simplices of $X$
for which there exists a vertex in $K$.

\begin{lemma}[The meet lemma]\label{meet-lemma}
  $$
  \Phi^{\cap K}_n \ = \ \sum_{p+m+q=n} \Phi^{\notin K}_p * \Phi^K_m * 
  \Phi^{\notin K}_q  .
  $$
\end{lemma}
\begin{proof}
  If an $n$-simplex $\sigma\in X_n$ has some vertex in $K$, then there is 
  a minimal vertex $x_p$ in $K$ and a maximal vertex $x_{p+m}$ in $K$. 
  (They might coincide, which would be the case $m=0$.) Since $K \to X$ is 
  full, the edge from $x_p$ to $x_{p+m}$ is contained in $K$, and since 
  $K\to X$ is also culf,  all intermediate vertices and faces belong to 
  $K$ too. So the simplex $\sigma$ necessarily has first $p$ edges not 
  belonging to $K$, then $m$ edges that belong to $K$, and finally $q$
  edges not belonging to $K$. So far we have referred to arbitrary 
  simplices, but we know from Lemma~\ref{nondeg by edges} that $\sigma$ 
  is nondegenerate if and only if its three parts are. Now we get the 
  formula at the level of the $\Phi$-functionals from the fundamental
  equivalence $\Phi_{p+q} = \Phi_p * \Phi_q$ (Lemma~\ref{Phipq}).
\end{proof}

\begin{lemma}[The S-lemma]\label{S-lemma}
  $$
  \Phi_s * \Phi_0^K = \sum_{p+i=s} \Phi^{\notin K}_p * \Phi^K_i  .
  $$
\end{lemma}
Note again that convolution from the right with $\Phi_0^K$ serves to 
impose the condition that the last vertex belongs to $K$. So intuitively 
the equation 
says that an $s$-simplex whose last vertex is in $K$ must 
have $p$ edges outside $K$ and then $i$ edges inside $K$. (It is because 
of convexity that there are no other possibilities.) Note also that specifying an 
$s$-simplex by imposing conditions on  specific edges like this is 
precisely what the convolution product expresses.

\begin{lemma}[The T-lemma]\label{T-lemma}
  $$
  \Phi_0^K * \Phi_t = \sum_{j+q=t} \Phi^K_j * \Phi^{\notin K}_q   .
  $$
\end{lemma}

This is the same, but for $t$-simplices whose zeroth vertex is in $K$.

\subsection{Crapo formula as a homotopy equivalence}

\begin{theorem}\label{crapo}
  For $K \subset X$ convex we have
$$
\mu^X = \mu^{X\shortsetminus K} + \mu^X *\zeta^K *\mu^X  .
$$

\end{theorem}

What it really means is
$$
\Phieven^X - \Phiodd^X =   (\Phieven^{X\shortsetminus K} - \Phiodd^{X\shortsetminus K})  
\ + \ (\Phieven^X - \Phiodd^X) * (\Phi_0^K+\Phi^K_1) * (\Phieven^X 
- \Phiodd^X)  ,
$$
and then expand and move all minus signs to the other side of the equation to
obtain finally the sign-free meaning of the theorem:
$$
\begin{array}{rr}
\Phieven^X\! &+ \ \Phieven^X \conv \Phi^K_1 \conv \Phiodd^X + 
\Phiodd^X\conv\Phi^K_1\conv\Phieven^X \\[8pt]
\Phiodd^{X\shortsetminus K}\! &+ \
\Phieven^X\conv\Phi^K_0\conv\Phiodd^X + 
\Phiodd^X\conv\Phi^K_0\conv\Phieven^X
\end{array}
\ \
=
\ \
\begin{array}{rr}
\Phieven^{X\shortsetminus K} \! &+ \ \Phieven^X\conv\Phi^K_0\conv\Phieven^X + 
\Phiodd^X\conv\Phi^K_0\conv\Phiodd^X
 \\[8pt]
\Phiodd^X\! &+ \
\Phieven^X\conv\Phi^K_1\conv\Phieven^X + 
\Phiodd^X\conv\Phi^K_1\conv\Phiodd^X  .
\end{array}
$$
This is the explicit homotopy equivalence we establish.

\begin{blanko}{Scholium.}
  $$
  \Phi^X_n = \Phi^{X\shortsetminus K}_n + \Phi^{\cap K}_n  .
  $$
\end{blanko}

From the viewpoint of $X$, this is clear: a nondegenerate $n$-simplex in $X$ either has a 
vertex in $K$ or it does not have a vertex in $K$. A simplex in $X$ without a vertex 
in $K$ is the same thing as a simplex in $X\shortsetminus K$.  We can 
therefore interpret the symbol as $g\lowershriek (\Phi^{X\shortsetminus 
K}_n)$.

\begin{lemma}[Key Lemma]
  $$
  \Phi^{\cap K}_n \ + \ \sum_{s+1+t=n} \Phi^X_s * \Phi^K_1 * \Phi^X_t
  \ = \
  \sum_{s+t=n} \Phi^X_s * \Phi^K_0 * \Phi^X_t  .
  $$
\end{lemma}
This has the same overall shape as the K-lemma, but note that unlike in the 
K-lemma,
the terms on the LHS do not simply identify with those on the RHS.

\begin{proof}
  We expand the term $\Phi^{\cap K}_n$ using the meet lemma; we expand the 
  $s$-indexed terms using the S-lemma; we expand the $t$-indexed terms 
  using the T-lemma. The claim thus becomes
  $$
  \sum_{\mathclap{p+m+q=n}} \Phi^{\notin K}_p * \Phi^K_m *
  \Phi^{\notin K}_q 
  \ + \ 
  \sum_{\mathclap{p+i+1+j+q=n}} \Phi^{\notin K}_p * \Phi^K_i * \Phi^K_1 * \Phi^K_j *
  \Phi^{\notin K}_q
  \ = \
  \sum_{\mathclap{p+i+j+q=n}} \Phi^{\notin K}_p * \Phi^K_i * \Phi^K_0 * \Phi^K_j *
  \Phi^{\notin K}_q   .
  $$
But this equation is precisely the K-lemma convolved with $\Phi^{\notin K}_p $
from the left and with $\Phi^{\notin K}_q $ from the right.
\end{proof}

The following figure illustrates the relationship among the indices:
\begin{center}
  
  \colorlet{col1}{blue!8}
\colorlet{col2}{green!12}
\colorlet{col3}{red!8}

\begin{tikzpicture}
\footnotesize

	\draw (-3.9, 0.6) --++ (0, 0.1) --++ (3.6, 0) --++ (0, -0.1);
	\node at (-2.1,0.9) {\footnotesize $s$};
	\draw (3.9, 0.6) --++ (0, 0.1) --++ (-3.6, 0) --++ (0, -0.1);
	\node at (2.1,0.95) {\footnotesize $t$};

	\fill[col3] (-2.0,-0.25) rectangle ++(-1.9,0.5);
	\node at (-2.9,0) {\footnotesize $p$};
  	\fill[col1] (-0.3,-0.25) rectangle ++(-1.6,0.5);
	\node at (-1.2,0.04) {\footnotesize $i$};
	\node at (0,0) {\tiny $(0|1)$};
  	\fill[col1] (0.3,-0.25) rectangle ++(1.6,0.5);
	\node at (1.2,0) {\footnotesize $j$};
  	\fill[col3] (2.0,-0.25) rectangle ++(1.9,0.5);
	\node at (2.9,0) {\footnotesize $q$};
	
	\draw (-1.9, -0.6) --++ (0, -0.1) --++ (3.8, 0) --++ (0, 0.1);
	\node at (0.0,-0.9) {\footnotesize $m$};

	\draw (-3.9, -1.2) --++ (0, -0.1) --++ (7.8, 0) --++ (0, 0.1);
	\node at (0.0,-1.6) {\footnotesize $n$};

\end{tikzpicture}
\end{center}

\begin{proof}
  Here is an alternative proof: Start with the K-lemma:
    $$
    \Phi_m + \sum_{i+1+j=m} \Phi_i * \Phi_1 * \Phi_{j} \ = \
  \sum_{i+j=m} \Phi_i * \Phi_j  .
  $$
  Now convolve with $\Phi^{\notin K}_p $
from the left and with $\Phi^{\notin K}_q $ from the right, and sum to $n$ 
to obtain
  $$
  \sum_{\mathclap{p+m+q=n}} \Phi^{\notin K}_p * \Phi^K_m *
  \Phi^{\notin K}_q 
  \ + \ 
  \sum_{\mathclap{p+i+1+j+q=n}} \Phi^{\notin K}_p * \Phi^K_i * \Phi^K_1 * \Phi^K_j *
  \Phi^{\notin K}_q
  \ = \
  \sum_{\mathclap{p+i+j+q=n}} \Phi^{\notin K}_p * \Phi^K_i * \Phi^K_0 * \Phi^K_j *
  \Phi^{\notin K}_q  .
  $$
  The first sum on the left gives $\Phi^{\cap K}_n$ by the meet lemma.
  In the other sums, apply the S-lemma to the $s$-indexed terms and apply 
  the T-lemma to the $t$-indexed terms. Altogether we arrive at
    $$
  \Phi^{\cap K}_n
  \ + \ 
  \sum_{s+1+t=n} \Phi_s * \Phi^K_1 * \Phi_t 
  \ = \
  \sum_{s+t=n} \Phi_s * \Phi^K_0 * \Phi_t   ,
  $$
  which is what we wanted to prove.
\end{proof}

\subsection{Finiteness conditions and cardinality}
\label{sub:fin}

In order to take homotopy cardinality to deduce results at the level of 
$\Q$-algebras, some finiteness conditions must be imposed.
First of all, for the incidence (co)algebra of $X$ to admit a cardinality, 
$X$ should be locally 
finite, meaning that all active maps are finite.
Second, for the general M\"obius inversion formula to admit a cardinality, we must 
ask that for each $1$-simplex $f$, there are only finitely many non-degenerate 
$n$-simplices (any $n$) with long edge $f$. This is the M\"obius 
condition for decomposition spaces~\cite{Galvez-Kock-Tonks:1512.07577}.

We should only remark that if $X$ is a M\"obius decomposition space, 
and if $K\subset X$ is a convex subspace then also $K$ is M\"obius (this 
follows since anything culf over a M\"obius decomposition space is 
M\"obius again. We should also check whether it is true that for any full 
inclusion $Y \to X$, we have that $Y$ is M\"obius again.

\bigskip

Note that for an ikeo map to admit a cardinality (which will then be an 
algebra homomorphism) it must be a finite map. In the present case this is 
OK since the maps are even mono.

\bigskip

Recall that a simplicial space $Y$ is locally finite if all active maps are 
finite. (Note that in \cite{Galvez-Kock-Tonks:1512.07577} it was also 
demanded that $Y_1$ be locally finite, but this has turned out not to be 
necessary. The following results remain true with this extra condition, 
though.)

\begin{lemma}
  If $F: Y \to X$ is a full inclusion of simplicial spaces, and if $X$ is 
  locally finite, then also $Y$ is locally finite.
\end{lemma}

\begin{proof}
  Let $g: Y_n \to Y_1$ be the unique active map. We need to show that 
  the fibre over any $a\in Y_1$ is finite. In the cube diagram
    \[
  \begin{tikzcd}[column sep={6em,between origins}, row sep={3.5em,between
	origins}]
  & 1 \ar[ldd]\ar[d, "\name{a}"] & (Y_n)_a \ar[d]\ar[l]  \\
  & Y_1 \ar[ldd] & Y_n \ar[ldd] \ar[l, pos=0.7, "g"'] \\
  1  \ar[d, "\name{F(a)}"'] & (X_n)_{F(a)} \ar[d] \ar[l,crossing over] & \\
  X_1 & \ar[l] X_n \ar[l, "g"] &
  \ar[from=1-3,to=3-2,crossing over]
  \end{tikzcd}
  \]
  the back and front faces are pullbacks by definition of the fibres we 
  are interested in.  The left-hand face is a pullback since $Y_1 \to X_1$ 
  is mono. By the prism lemma it now follows that also the right-hand face is 
  a pullback. Finally we see that $(Y_n)_a \to (X_n)_{F(a)}$ is mono 
  because it is  a pullback of
  $Y_n \to X_n$, which is mono since $F$ is a full inclusion.
  Since $(X_n)_{F(a)}$ is finite,
  it follows that $(Y_n)_a$ is finite.
\end{proof}

Recall (from~\cite[\S 6]{Galvez-Kock-Tonks:1512.07577}) that the {\em
length} of a $1$-simplex $a\in Y_1$ is defined as the dimension of the
biggest {\em effective} $n$-simplex $\sigma$ with long edge $a$. Effective
means that all principal edges are nondegenerate. For decomposition spaces,
and more generally for so-called stiff simplicial
spaces~\cite{Galvez-Kock-Tonks:1512.07577}), this is equivalent to $\sigma$
being nondegenerate.

\begin{lemma}
  Let $Y\to X$ be a conservative simplicial map between locally finite 
  simplicial spaces. If $X$ is (complete and)
  of locally finite length, then also $Y$ is (complete and) of locally 
  finite length.
\end{lemma}

\begin{proof}
  Note first that if $X$ is complete, then so is $Y$, since the map is
  conservative. If $Y$ were not of locally finite length, that would
  mean there is a $1$-simplex $a\in Y_1$ for which $(\nondeg Y_n)_a$ is 
  nonempty for all $n$. (This is not the definition of locally finite 
  length, but for locally finite simplicial spaces this is equivalent.)
  But each $\sigma\in (\nondeg Y_n)_a$ witnessing this nonemptiness
  is sent to $f\sigma \in (\nondeg X_n)_{fa}$ witnessing also infinite
length of $fa$. Note that effective simplices are preserved, as a
consequence of being conservative.
\end{proof}

\begin{cor}
  If $Y \to X$ is a full inclusion of simplicial spaces, and if $X$ is a 
  M\"obius decomposition space, then also $Y$ is a M\"obius decomposition 
  space.
\end{cor}

With these preparations we see that in the Crapo formula, if just the 
ambient decomposition space $X$ is M\"obius, then also $K$ and 
$X\shortsetminus K$ are M\"obius, so that all the objects in the formula 
admit a cardinality. The formula therefore holds at the level of 
$\Q$-vector spaces.

\footnotesize

\noindent
\noindent
Universidad de M\'alaga, IMTECH-UPC, and Centre de Recerca Matemàtica\\
E-mail address: \texttt{imma.galvez@uma.es}

\bigskip

\noindent
University of Copenhagen, Universitat Aut\`onoma de Barcelona, and Centre de Recerca Matem\`atica  \\
E-mail address: \texttt{kock@math.ku.dk}

\bigskip

\noindent
Universidad de M\'alaga\\
E-mail address: \texttt{at@uma.es}
\end{document}